\documentclass[12pt]{article}
\usepackage{e-jc}
\usepackage{amsthm,amsmath,amssymb}
\theoremstyle{plain}
\newtheorem{Thm}{Theorem}[section]
\newtheorem{Cor}[Thm]{Corollary}

\newtheorem{lemma}[Thm]{Lemma}
\newtheorem{Prop}[Thm]{Proposition}
\newtheorem{Def}[Thm]{Definition}

\newtheorem{remark}[Thm]{Remark}
\newtheorem{example}[Thm]{Example}
\newtheorem{Ques}{Question}[section]

\newcommand{\gbar}{\overline{G}}
\newcommand{\Hbar}{\overline{H}}
\newcommand{\Fbar}{\overline{F}}

\title{Nordhaus-Gaddum Theorem for the \\ Distinguishing Chromatic Number}

\author{Karen L. Collins\\
\small Dept. of Mathematics and Computer Science\\
\small Wesleyan University\\
\small Middletown CT 06459-0128\\
\small\tt kcollins@wesleyan.edu\
\and
Ann Trenk\\
\small Department of Mathematics\\
\small Wellesley College\\
\small Wellesley MA 02481\\
\small\tt atrenk@wellesley.edu
}

\date{\dateline{March 9, 2012}{March 10, 2012}}

\begin{document}
\maketitle

\begin{abstract}  Nordhaus and Gaddum proved, for any graph $G$, that  $\chi(G) + \chi(\overline{G}) \leq n + 1$, where $\chi$ is the chromatic number and $n=|V(G)|$.  Finck characterized the class of  graphs that satisfy equality in this bound.  In this paper, we provide a new characterization of this class of graphs, based on vertex degrees, which yields a new polynomial-time recognition algorithm and efficient computation of the chromatic number of graphs in this class.   Our motivation comes from our theorem that generalizes the Nordhaus-Gaddum theorem to the distinguishing chromatic number:  for any graph $G$, $\chi_D(G) +\chi_D(\overline{G})\leq n+D(G)$. Finally, we characterize those graphs that achieve  equality in the sum upper bounds simultaneously for both the chromatic number and for our distinguishing chromatic number analog of the Nordhaus-Gaddum inequality. 

\end{abstract}


\section{Introduction}

We provide a generalization of the classic Nordhaus and Gaddum Theorem for the chromatic number to the distinguishing chromatic number.  First, we recall their theorem, which gives bounds on the sum and the product of the chromatic number of a graph with that of its complement.  We write $\chi(G)$ for the chromatic number of graph $G$, and 
for the complement of graph $G$ we write $\gbar$. 
The upper bound in (2) below was proved by Zykov~\cite{Zy49} and the remaining three inequalities were proved by Nordhaus and Gaddum~\cite{NoGa56}.  
\begin{Thm} 
\smallskip 

If $G$ is a graph with $|V(G)| = n$ and $\chi(G)$ is the chromatic number of $G$, then
\begin{equation}
2 \sqrt{n} \le \chi(G) + \chi(\gbar) \le n + 1.
\end{equation}
\begin{equation}
n \le \chi(G) \chi(\gbar) \le \left(\frac{n+1}{2}\right)^2.
\end{equation}

\label{orig-NG}
\end{Thm} 
\noindent Finck \cite{Fi66} characterized the graphs that achieve equality for each of the four bounds in Equations (1.1) and (1.2). 

{\rm A labeling (or coloring) of the vertices of a graph $G$, $h :V(G) \rightarrow
\{1,\ldots,r\}$, is said to be \emph{$r$-distinguishing} (or
just \emph{distinguishing}) if the only automorphism of the graph that preserves all of the vertex labels is the identity.  The \emph{distinguishing number of $G$,} denoted by $D(G)$,
is  defined as the minimum number $r$ so that $G$ has an $r$-distinguishing labeling.
Albertson and Collins  study the distinguishing number in \cite{AlCo96}   and  subsequently 
other authors have studied the distinguishing number of  graphs and of other structures, see 
for example \cite {AlBo06, ArChDe08,  BaCa11, ElSc11,  ImKlTr07, Ko08, LaNgSa10,Ty04, WoZh09}, and many others.  

The automorphism group of a graph is the same as the automorphism group of its complement, hence we get the following remark.

\begin{remark} For any graph $G$, we have \label{dgbar} $D(G)=D(\gbar)$.
\end{remark}

In \cite{CoTr06} we introduce the distinguishing chromatic number of a graph $G$, denoted by $\chi_D(G)$, that requires the coloring to be proper as well as distinguishing.  Together with Hovey we explored the distinguishing chromatic number from the perspective of group theory in \cite{CoHoTr09}.  The subject has received considerable attention from others, who considered the distinguishing chromatic number in \cite{ChHaKa10}, and others both the distinguishing number and the distinguishing chromatic number \cite{Ch09, LaBh09, Tu11, WaZh07}.   In this paper we ask whether there is a version of Theorem \ref{orig-NG} for the distinguishing chromatic number. 
\begin{Def}
{\rm A labeling (or coloring) of the vertices of a graph $G$, $h :V(G) \rightarrow
\{1,\ldots,r\}$, is said to be \emph{proper $r$-distinguishing} (or
just \emph{proper distinguishing}) if it is a proper labeling (i.e.,
coloring)  of the
graph and the only automorphism of the graph that preserves all of the vertex
labels is the identity.  The \emph{distinguishing chromatic number} of a graph $G$, denoted by
$\chi_{D}(G)$, is the minimum $r$ such that $G$ has a proper
$r$-distinguishing labeling.}
\end{Def}
We note that the two lower bounds from Theorem~\ref{orig-NG}  are still valid for the distinguishing chromatic number since $\chi(G) \le \chi_D(G)$ for all graphs $G$.    Thus for any graph $G$ with $n=|V(G)|$ we have: 

\begin{equation}2 \sqrt{n} \le \chi_D(G) + \chi_D(\gbar) \mbox{ and} 
\end{equation}
\begin{equation} n \le \chi_D(G)\cdot \chi_D(\gbar)
\end{equation}
For any graph $G$ with $D(G)=1$, we have $\chi(G)=\chi_D(G)$, and so any graph $G$ with $D(G)=1$ that  satisfies equality in one of the lower bounds of Equations (1.1) and (1.2)  will  be an example of a graph for which the corresponding bound in Equations (1.3) and (1.4) is tight.  Finck's constructions \cite{Fi66} of such graphs include examples with $D(G)
=1$.  Cavers and Seyffarth provide further examples in  \cite{CaSe}.

Before concluding this section, we present some background definitions and Brooks' Theorem.  We use $|S|$ to denote the size of set $S$ and 
$\Delta (G)$ to denote the largest vertex degree in graph $G$.  %We define the join of two graphs, $G\vee H$, to be the graph with vertex set $V(G)\cup V(H)$ and edge set $E(G)\cup E(H)\cup \{\{x,y\}\;|\;x\in V(G), y\in V(H)\}$. 
The independent set with $s$ vertices is denoted  by $I_s$.  For a vertex $u\in V(G)$, we let $N(u)$ be the set of neighbors of $u$ in $G$.  
We will routinely use  $\gbar - v$ in place of $\overline{G-v}$ and it is easy to see that these are equivalent.
 If $S$ is a set of vertices in $G$, we write $G[S]$ to denote the subgraph  induced  in $G$ by $S$.    We write $Aut(G)$ for the group of all automorphisms of the graph $G$.
 We say that  graph $H$ is \emph{color-critical} if $\chi(H-x) < \chi(H)$ for every vertex $x \in V(H)$.  We will also need Brooks' Theorem:

\begin{Thm} 
(Brooks [1941]) If $G$ is a connected graph other than a complete graph or an odd cycle, then $\chi(G)\leq \Delta(G)$. 
\label{brooks-thm}
\end{Thm}

In this paper we revisit the Nordhaus-Gaddum inequalities (Theorem~\ref{orig-NG}) and the classes of graphs for which the upper bound in Equation~(1.1) is tight.  In Section~2 we give analogues of the upper bounds in Equations~(1.1) and (1.2) for the distinguishing chromatic  number.  In Section~3 we  give a new characterization of those graphs that achieve equality for the upper bound in Equation~(1.1), based on vertex degrees.  Our characterization leads to a new polynomial-time recognition algorithm for this class and efficient computation of the chromatic number of graphs in this class.  In Section~4 we characterize those graphs that achieve  equality in the upper bound of Equation~(1.1) and our distinguishing chromatic number analog of this Nordhaus-Gaddum inequality.

\section{The Nordhaus Gaddum inequalities for $\chi$ and  $\chi_D$}

Nordhaus and Gaddum \cite{NoGa56} describe three classes of graphs to illustrate that their bounds are tight.    The first class is the complete graphs, which are tight for the upper bound in Equation~(1.1) and the lower bound in Equation~(1.2); next is the complete multipartite graphs with $q$ parts, each of size $q$, which are tight for the lower bounds in Equation~(1.1) and Equation~(1.2); and third,  the disjoint union of a complete graph and an independent set with one fewer vertex, $K_n+I_{n-1}$ which are tight for the upper bounds in Equation~(1.1) and Equation~(1.2).   They note that it is not possible to satisfy the lower bound in Equation~(1.1) and the upper bound in Equation~(1.2) simultaneously.  In Table \ref{tab:NG-table}, we record the values for the distinguishing number and the distinguishing chromatic number for these examples.

\begin{table}[htbp] 

\begin{tabular}{|l||c|c|c|c|c|c|} \hline 
&&&&& &\\
$G$ & $|V(G)|$& $\chi_D(G)$ &$\chi_D(\gbar)$& $\chi_D(G)+\chi_D(\gbar) $& $\chi_D(G)\cdot \chi_D(\gbar)$& $D(G)$\\ \hline 
&&&&&&\\
$K_n$ &$n$ &$n$ & $n$ & $2n$ & $n^2$& $n$\\ \hline
&&&&&&\\
$K_{\underbrace{q,q,\ldots,q}_{\mbox{\footnotesize $q$ times}}}$&$q^2$ &$q^2$ & $q+1$& $q^2+q+1$ & $q^2(q+1)$&$q+1$ \\ \hline
%$s=q_1=q_2=\cdots=q_s$ & $q$ & $q$ & $2q$& $q^2$\\ \hline
&&&&&&\\
$K_t+I_{t-1}$& $2t-1$ & $t$ & $2t-1$ & $3t-1$&$(2t-1)t$ &$t$\\ \hline
\end{tabular}
\bigskip
\caption{\small Examples from Nordhaus and Gaddum \cite{NoGa56}, together with their distinguishing chromatic numbers. }
\label{tab:NG-table} 
\end{table}

The examples in Table \ref{tab:NG-table} make clear that we will need to increase the upper bounds in Equations~(1.1) and (1.2) in order to prove analogues for the distinguishing chromatic number.  Each of these examples in the table is a complete multipartite graph ($K_n$ and $K_{q,q,\ldots,q}$) or the complement of a complete multipartite graph ($K_t+I_{t-1}$).  
 Collins and Trenk \cite{CoTr06} have shown that complete multipartite graphs are exactly the graphs $G$ for which $\chi_D(G)=|V(G)|$, that is,  the graphs with the largest possible distinguishing chromatic number.  Note that the distinguishing number of each graph in the table is equal to either its distinguishing chromatic number or the distinguishing chromatic number of its complement.  This leads us to our first step in finding an appropriate generalization of the Nordhaus-Gaddam theorems, which is to consider the distinguishing chromatic number of the complements of complete multipartite graphs.  
 
\begin{Prop} \label{cmg-gbar} 
Let $G$ be a complete multipartite graph.  Then $\chi_D(\gbar)=D(\gbar)$. 
\end{Prop}

\begin{proof}  By the definition of $\chi_D$ and $D$, the inequality  $\chi_D(\gbar)\geq D(\gbar)$ holds for all graphs $G$. Since $G$ is a complete multipartite graph, we know that  $\gbar$ is a collection of disjoint complete graphs.   Let $\phi:V(\gbar)\rightarrow \{1,2, \ldots, D(\gbar)\}$ be a distinguishing labeling of $\gbar$.  Let $u,v\in V(\gbar)$ be adjacent in $\gbar$.  Then $u,v$ are both in the same complete subgraph of $\gbar$.  The automorphism of $\gbar$ that switches $u$ and $v$ and fixes all other vertices must not preserve labels, so $\phi(u)\neq \phi(v)$.   Thus, $\phi$ is both proper and distinguishing, so $\chi_D(\gbar)\leq D(\gbar)$ and $ \chi_D(\gbar)=D(\gbar)$.
\end{proof}

This suggests a natural generalization of the Nordhaus-Gaddum bound. 
Theorem~\ref{NG-new}  presents an upper bound generalizing Equation~(1.1), and Corollary~\ref{NG-cor} gives the resulting upper bound generalizing Equation~(1.2).
Note that the analogous lower bounds were presented in Equations (1.3) and (1.4).

\begin{Thm}
If $G$ is a graph  with $n = |V(G)|$  then \[   \chi_D(G) + \chi_D(\gbar) \le n + D(G)\]

\label{NG-new}
\end{Thm}

\begin{proof}
Fix a distinguishing coloring of graph $G$ using   colors  in the set \[C=\{1,2,3,\ldots, D(G)\}\]  This simultaneously provides a distinguishing coloring of $\gbar$.      For each $i \in C$, we let $V_i$ be the vertices of color $i$, and let  $G_i = G[V_i]$, thus  $\overline{G_i} = \gbar[V_i]$.    By Theorem~\ref{orig-NG}, we know $\chi(G_i) + \chi(\overline{G_i} ) \le |V_i| + 1 $ for each $i \in C$.  Thus we may recolor the graph $G_i$ and separately recolor the graph  $\overline{G_i}$ using   $|V_i| + 1$ new colors so that  both new colorings are proper.    We do this for each $i \in C$ using a new set of $|V_i| + 1 $ colors for each $i$.    The result is a   coloring of $G$ and a   coloring of $\gbar$ using a total of $\sum_{i=1}^{D(G)}(|V_i| + 1) = |V(G)| + D(G) = n + D(G)$ colors.    By construction, these   colorings of $G$ and $\gbar$ are proper.  Moreover,  we show they are distinguishing.  Suppose there were a non-trivial automorphism $\sigma$  of $G$ that preserved colors.  Since a new set of colors is used for each $i$, we know that $\sigma$ must  preserve membership in $V_i$ for each $i$.  However, the original coloring was distinguishing, so the only  automorphism  of $G$ that preserves membership in $V_i$ for each $i$ must be the identity.
 
\end{proof}

\begin{Cor}
If $G$ is a graph  with $n = |V(G)|$  then 
 $  \chi_D(G) \chi_D(\gbar) \le \left(\frac{n+D(G)}{2}\right)^2.$ 
 
\label{NG-cor}
\end{Cor}

\begin{proof}
 We follow the proof given in  \cite{NoGa56}.    For all real numbers $x,y $ we know $0 \le (x-y)^2  $  and thus 
$ 4xy \le (x+y)^2$ and $xy \le  (\frac{x+y}{2} )^2$.   Substitute   $x = \chi_D(G)$ and $y = \chi_D(\gbar)$ into this last inequality and then apply Theorem~\ref{NG-new} to finish the proof.
\end{proof}

Theorem \ref{NG-new} is robust, and in Proposition \ref{prop-cD} we extend it to any group action on our graph $G$.   

\begin{Def} \rm \label{DG}  Let $G$ be a graph  and let $\Gamma$ be a subgroup of $Aut(G)$.  The {\it  distinguishing number of $G$ with respect to $\Gamma$,} denoted by $D^{\Gamma}(G)$, is the minimum number of colors needed to color the vertices of $G$ so that  no non-identity element of $\Gamma$  preserves the colors.  
\end{Def}

\begin{Def} \rm Let $G$ be a graph  and let $\Gamma$ be a subgroup of $Aut(G)$.  The {\it  distinguishing chromatic number of $G$ with respect to $\Gamma$,} denoted by $\chi_D^{\Gamma}(G)$, is the minimum number of colors needed to color the vertices of $G$ so that the coloring is proper, and no non-identity element of $\Gamma$  preserves the colors.  
\end{Def}
 
\begin{Prop} \label{prop-cD} If  $G$ be a graph  with $n = |V(G)|$  and  $\Gamma$ is a subgroup of $Aut(G)$, 
then 
 \[  \chi_D^{\Gamma}(G) + \chi_D^{\Gamma}(\gbar) \le n + D^{\Gamma}(G).\]

\label{NG-new-r}
\end{Prop}

\begin{proof}
The proof is similar to the proof of Theorem~\ref{NG-new}.
\end{proof}

We now turn to the question of characterizing those graphs that achieve equality in the upper bounds of identity of Theorem~\ref{orig-NG} and the related question of characterizing the analogous graphs for Theorem~\ref{NG-new}.

\begin{Def}  {\rm
A graph $G$  with $|V(G)| = n$ is an  \emph{NG-graph} if it satisfies $\chi(G) + \chi(\gbar) = n + 1$, and is an \emph{NGD-graph} if it satisfies $   \chi_D(G) + \chi_D(\gbar) = n + D(G)$. }
\label{NG-def}
\end{Def}

Proposition \ref{cmg-gbar} shows that all complete multipartite graphs, including $K_n$ and $\overline{K_t+I_{t-1}}$, are NGD-graphs.  However, they are not all NG-graphs, see Table \ref{small-table}  at the beginning of Section~4. 

\begin{Cor} 
\label{H-F-cor}
If $G$ is an NGD-graph with a fixed distinguishing coloring  using $D(G)$ colors, then each color class induces  an NG-graph.  
\end{Cor}

\begin{proof}
Let $G$ be an NGD-graph and fix a distinguishing coloring of $G$ using $D(G)$ colors:  $1,2,3, \ldots, D(G)$.    Let $V_i$ be the vertices of color $i$ and let $G_i = G[V_i]$.  If $\chi(G_i) + \chi(\overline{G_i}) < |V_i| + 1$ for any $i$, then following the proof of Theorem~\ref{NG-new}, we would have a distinguishing coloring of $G$ and $\gbar$ using fewer than $n + D(G)$ colors.  This contradicts the assumption that $G$ is an NGD-graph.
\end{proof}

\section{Characterizing NG-graphs}

In this section we focus on the ordinary chromatic number $\chi$ and the inequality
 $\chi(G) + \chi(\gbar) \le n + 1$ of Theorem~\ref{orig-NG}.   Our  main result of this section is a characterization of NG-graphs, that is, the graphs that satisfy this with equality.  Our characterization leads to a polynomial-time recognition algorithm for NG-graphs and an efficient computation of the  chromatic number of NG-graphs.

Finck~\cite{Fi66} characterizes the graphs that achieve equality for each of the four inequalities in Theorem~\ref{orig-NG}.  His characterizations involve arrays and in the case of NG-graphs,  he gives an induction proof based on $\chi(H)$ for certain induced subgraphs $H$ of $G$.  This proof is not constructive, nor does it lead to a polynomial-time algorithm for recognizing whether a given graph is an NG-graph.  Starr and Turner~\cite{StTu08} give a characterization of NG-graphs that is simpler to state but relies explicitly on $\chi(G)$ and thus also can not be used to recognize NG-graphs in polynomial-time.  Our characterization depends on partitioning vertices accordinng to their degree and leads to a polynomial-time algorithm to determine whether a graph is an NG-graph and if so to find its chromatic number.
%%%%%%%%%%%

\begin{Def}\rm 
If $G$ is an NG-graph, then the \emph{ABC-partition} of $V(G)$ is as follows:

 $A_G = \{v \in V(G): deg(v) = \chi(G) -1\}$
 
 $B_G = \{v \in V(G): deg(v) > \chi(G) -1\}$
 
 $C_G = \{v \in V(G): deg(v) < \chi(G) -1\}$
 
 When it is unambiguous, we write $A=A_G$, $B=B_G$, $C=C_G$.

\label{ABC-def}
\end{Def}

The following theorem characterizes NG-graphs and Figure~\ref{NG-fig}  illustrates the three possible forms.
\begin{Thm} \label{NG-charac} 
 A graph $G$  is an NG-graph if and only if when its vertex set is partitioned $V(G) = A_G \cup B_G \cup C_G$  we have
 
 (i) $A_G \neq \emptyset$ and  $G[A_G]$ is a clique, an independent set, or a 5-cycle
 
 (ii) $G[B_G]$ is a clique.
 
 (iii)  $G[C_G]$ is an independent set
 
 (iv)  $uv \in E(G)$ for all $u \in A_G$, $v \in B_G$ 
 
 (v)  $uw \not\in E(G)$ for all $u \in A_G$, $w \in C_G$.
 
\label{NG-equal-thm}
\end{Thm}

\begin{figure}[htbp]
\begin{center}

\begin{picture}(200,100)(0,30)
\thicklines
\put(50,100){\circle{40}} 
\put(100,50){\circle{40}} 
\put(150,100){\circle{40}} 

 \put(15,110){$A$}
 \put(175,110){$B$}
 \put(125,30){$C$}
 
 \put(95,47){$I_c$}
  \put(143,97){$K_b$}
  \put(43,110){$K_a,$}
   \put(33,95){$I_a$ or }

 \put(56,95){$C_5$}
 
 \put(70,100){\line(1,0){60}}
 \put(90,105){all}

\put(56,80){\line(1,-1){24}}
\put(47,60){none}

\put(144,80){\line(-1,-1){24}}
\put(140,60){?}
  
\end{picture}

\end{center}
\caption{The forms of an NG-graph}
 
\label{NG-fig}
\end{figure}
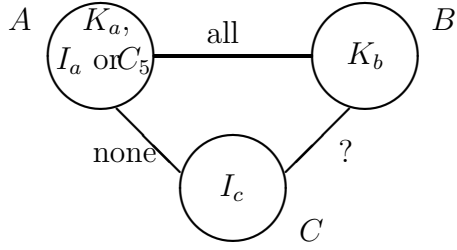

\begin{proof}
\noindent
($\Longleftarrow$) Let $A=A_G, B=B_G$, and $C=C_G$. 
In the case that  $G[A]$ is a clique, we can write  
$G[A] = K_a$,  $G[B] = K_b$, and $G[C] = I_c $ for some integers $a,b,c$ where $a \ge 1$.   We observe that  $\chi(G) = a+b$ since we need $a+b$ colors for $A \cup B$ and we may reuse a color from $A$ for all vertices in $C$.
In addition, $\chi(\gbar) = c+1$  since we need   $c$ colors for $C$ and one new color for $A \cup B$.  We have  $\chi(G) + \chi(\gbar) = a+b+c + 1 = n+1$,  so $G$ is an NG-graph.

In the case that  $G[A]$ is an independent set, let $G[A] = I_a$,  $G[B] = K_b$, and $G[C] = I_c$ where $a,b,c$ are integers with $a \ge 1$.   Then $\chi(G) = b+1$ ($b$ colors for $B$, one for the rest of the vertices), 
$\chi(\gbar) = a+c$   ($a+c$ colors for $A \cup C$, then reuse   a color from $A$ for all of $B$).  Thus  $\chi(G) + \chi(\gbar) = b + 1 + a + c  =  n+1$,   so $G$ is an NG-graph.

Finally, if $G[A]$ is a 5-cycle, choose let
$G[A] = C_5$,  $G[B] = K_b$, and $G[C] = I_c $ where $b$ and $c$ are integers.   Then $\chi(G) = b+3$ (reuse a color from $A$ for all vertices of $C$),
$\chi(\gbar) = c+3$ (reuse a color from $A$ for all vertices of $B$), and $n = 5 + b + c$,  so $G$ is an NG-graph.

We prove the converse of Theorem~\ref{NG-equal-thm} after a series of lemmas.  

\begin{lemma}
If $x$ is a vertex in an NG-graph $H$ and $deg(x) > \chi(H) - 1$ then $deg_{\Hbar}(x) < \chi(\Hbar) - 1$.  Moreover, $x$ is color-critical in $H$ but not in $\Hbar$.
\label{idea-lem}
\end{lemma}
\begin{proof}
Using the assumption that $H$ is an NG-graph and the given degree condition we have
$\chi(\Hbar) - 1 = |V(H)| - \chi(H) > |V(H)| - (deg(x) + 1) = deg_{\Hbar}(x)$, thus 
$deg_{\Hbar}(x) < \chi(\Hbar) - 1 $.  A simple induction argument shows that $x$ is not color-critical in $\Hbar$ thus $\chi({\Hbar} -x) = \chi(\Hbar)$.  Applying Theorem~\ref{orig-NG} to the graph $H-x$ yields
$\chi(H-x) + \chi(\Hbar - x) \le (n-1) + 1 = n$.  If $\chi(H-x) = \chi(H)$ then 
$\chi(H) + \chi(\Hbar) \le n$, a contradiction because $H$ is an NG-graph.  Thus $\chi(H-x) < \chi(H)$ and $x$ is color-critical in $H$.
This completes the proof of all the assertions of Lemma~\ref{idea-lem}.
\end{proof}

\begin{lemma}
\label{B-edges}
Suppose $G$ is an NG-graph and define $A_G, B_G$ and $C_G$ as in Definition~\ref{ABC-def}.  
  For any $y \in B_G$ and any $x \in A_G \cup B_G$ we have $xy\in E(G)$.
 \end{lemma}

\begin{proof}
For a contradiction, assume $xy \not\in E(G)$.  By the definition of $B_G$ and Lemma~\ref{idea-lem} we know $y$ is color-critical in $G$ and thus $\chi(G-y) = \chi(G) -1$.  Then 
$$deg_{G-y}(x) = deg_G(x) \ge \chi(G) - 1 = \chi(G-y) > \chi(G-y)-1.$$ 
By Lemma~\ref{idea-lem}, $x$ is color-critical in $G-y$, so $\chi(G-y-x) = \chi(G-y) - 1 = \chi(G) - 2$.  Now we can properly color $G$ using $\chi(G)-1$ colors by taking a proper coloring of $G-x-y$ using $\chi(G) -2$ colors and using one additional color for $x$ and $y$, a contradiction.
\end{proof}

\begin{lemma}
\label{C-lem}
Suppose $G$ is an NG-graph and   $A=A_G,  \  B=B_G$ and $C=C_G$ are defined as in Definition~\ref{ABC-def}.  
Then $G[C]$ is an independent set and there are no edges in $G$ between vertices of $A$ and vertices of $C$.
 \end{lemma}

\begin{proof}
Given that $G$ is an NG-graph,  Definition~\ref{NG-def} implies that $\gbar$ is also an NG-graph.  Let
$A'= A_{\gbar} = \{v:deg_{\gbar}(v) = \chi(\gbar) -1 \}$,   $B'= B_{\gbar}=\{v:deg_{\gbar}(v) > \chi(\gbar) -1 \}$,  and $C'= C_{\gbar} = \{v:deg_{\gbar}(v) < \chi(\gbar) -1 \}$.  Combining  Definition~\ref{NG-def}  with 
$deg_G(v) + deg_{\gbar}(v) = n-1$ yields $(deg_G(v) - \chi(G) + 1) + (deg_{\gbar}(v) - \chi(\gbar) + 1) = 0$.   Thus $A' = A$, $B' = C$ and $C' = B$.

Now Lemma~\ref{B-edges}  applied to $\gbar$ implies that $\gbar[B']$ is complete, hence $G[C]$ is an independent set.  For $x \in A'$ and $y \in B'$ the same lemma tells us that $xy \in E(\gbar)$ and thus for $x \in A$ and $y \in C$ we know $xy \not\in E(G)$.   \end{proof}
 
\begin{lemma}
\label{GA-lem}
 If $G$ is an NG-graph and  $A= A_G$  and $B=B_G$ are defined  as  in Definition~\ref{ABC-def}, and $A \neq \emptyset$ then $H=G[A]$ is an NG-graph and $\chi(G) = \chi(H) + |B|$.
 \end{lemma}

\begin{proof}
From the structure of graph $G$ deduced in Lemmas~\ref{B-edges} and \ref{C-lem}, we know $\chi(G) = \chi(H) + |B|$  and $\chi(\gbar) = \chi(\Hbar) + |C|$.  Thus 
$1 + |A| + |B| + |C| = 1+n = \chi(G) + \chi(\gbar) =  \chi(H) + \chi(\Hbar)  + |B| + |C|$, so 
$1 + |A| = \chi(H) + \chi(\Hbar)$ and $H$ is an NG-graph as desired.

\end{proof}

\begin{lemma}
\label{GA2-lem}
 If $G$ is an NG-graph and with $A=A_G,  B=B_G$ and $C=C_G$ are defined as in Definition~\ref{ABC-def},  
and $A \neq \emptyset$ then $H=G[A]$ is either a $5$-cycle, a complete graph, or an independent set.

 \end{lemma}

\begin{proof}
Let $v \in A$.  By our definition of the set $A$ we know $deg_G(v) = \chi(G) - 1$.  To compute the degree of $v$ in $H$, we subtract all neighbors in $B$ (and $C$),  and apply Lemmas~\ref{B-edges}, \ref{C-lem} and \ref{GA-lem} to obtain 
$deg_H(v) = \chi(G) -1 - |B| = \chi(H) - 1$.  Since this holds for all $v \in A$ we have 
$\Delta(H)  = \chi(H) -1$.  Now apply Brooks' Theorem (Theorem \ref{brooks-thm}). 

If $H$ is connected we conclude that $H$ is an odd cycle or a complete graph.  First consider the case $H = C_{2k+1}$.    If $k \ge 3$ we have $\chi(H) = 3$ and $\chi(\Hbar) = k+1$ so 
$\chi(H) + \chi(\Hbar) = k+4 \le k + (k+1) = 2k+1 = n$ which contradicts Lemma~\ref{GA-lem}.  If $k=1$ then $H$ is the complete graph $C_3$, hence we conclude that either $H$ is a 5-cycle ($k=2$) or $H$ is a complete graph.

Next we consider the case that $H$ is not connected.  Let $H_1$ be a component with maximum chromatic number and let $H_2$ be the rest of $H$, so $H_2 $ is not empty.  Then $\chi(H) = \chi(H_1)$ and every vertex in $H_1$ has degree $\Delta(H_1)=\Delta(H)=\chi(H)-1=\chi(H_1)-1$, so by Brooks' Theorem, $H_1$ is an odd cycle or a complete graph. 
%$1 + \Delta(H) =\chi(H)= \chi(H_1) \le 1 + \Delta(H_1) \le 1 + \Delta(H)$.  So equality holds throughout and $\Delta(H_1) = \Delta(H)$ and $\chi(H_1) = \Delta(H_1) + 1$.   Since $H_1$ is connected,   by Brooks' Theorem, $H_1$ is either a complete graph or an odd cycle.

\smallskip

\noindent
Case 1:    $H_1 = K_r$ for some $r \ge 1$.   Then $\chi(H) = \chi(H_1) = r$,  so $\chi(\Hbar)  = 1 + \chi(\overline{H_2})$.  By Lemma~\ref{GA-lem}, we know $H$ is an NG-graph, so 
$r  + |V(H_2)|  + 1 =  \chi(H) + \chi(\Hbar) = r + 1 + \chi(\overline{H_2})$.  Thus $ |V(H_2)| =  \chi(\overline{H_2})$ and $H_2$ is an independent set.   

By definition of $A$,  every vertex in $H$  has degree $\chi(H) - 1$.   Since $H_2 \neq \emptyset$, each vertex in $H_2$ has degree 0 in $H$, thus $\chi(H) - 1 = 0$ and $H$ is an independent set.

\smallskip

\noindent
Case 2:   $H_1 = C_{2k+1}$ for some $k \ge 2$.
In this case, $\chi(H) = \chi(H_1) = 3$ and $\chi(\Hbar) = \chi(C_{\overline{2k+1}}) + \chi(\overline{H_2}) = k+1 + \chi(\overline{H_2}) $.  Again, using Lemma~\ref{GA-lem}, we know $H$ is an NG-graph, so 
$  2k+1   + |V(H_2)| + 1  =  \chi(H) + \chi(\Hbar) =  3 + k+1 + \chi(\overline{H_2})$.   Thus $2-k =  |V(H_2)|  -  \chi(\overline{H_2}) \ge 0$ and $k \le 2$.  By the assumptions of this case, $k \ge 2$ so $k=2$ and the cycle $C_{2k+1}$ is a 5-cycle.    Substituting $k=2$ into the equation above yields $ |V(H_2)| =  \chi(\overline{H_2}) $, so $H_2$ is an independent set.   Now vertices in $H_1$ have degree 2  in $H$ and vertices in $H_2$ have degree 0 in $H$, contradicting the definition of $A$.
\end{proof}

Now we are ready to prove the other direction of Theorem~\ref{NG-equal-thm}.

\noindent
($\Longrightarrow$)
Suppose we are given an NG-graph $G$ and let $A = A_G$, $B = B_G$, $C=C_G$ as in Definition~\ref{ABC-def}.  We show that the partition   $V(G) = A \cup B \cup C$   satisfies $(i) - (v)$ of Theorem~\ref{NG-equal-thm}.  Conditions (ii) and (iv) follow from Lemma~\ref{B-edges} and conditions (iii) and (v) follow from Lemma~\ref{C-lem}.  If $A \neq \emptyset$, condition (i) follows from Lemma~\ref{GA2-lem}.  

Finally, we consider the case $A = \emptyset$, and show it leads to a contradiction.     Suppose $A = \emptyset$,   so  $V(G) =  B \cup C$  where  $G[B] = K_b$, and    $G[C] = I_c$.   First observe that  that $\chi(G) \ge b$ because $G$ contains $G[B] = K_b$.  In fact,  $\chi(G) =  b$ since, by definition, each vertex in $C$ has degree at most   $ \chi(G) - 1 \ge b-1$,  and since $\chi(G) - 1 \ge b-1$,   each vertex in $C$ can be colored using one of the $b$ colors not appearing among its neighbors.    Similarly, we show $\chi(\gbar) = c$.  We know $\gbar[B] = I_b$ and $\gbar[C] = K_c$, so $\chi(\gbar) \ge c$ and each vertex of $C$ requires its own color.  By 
 Lemma~\ref{idea-lem}, we know each vertex $x \in B$ has $deg_{\gbar}(x) \le \chi(\gbar) - 1$, so no additional colors are needed for vertices in $B$ and $\chi(\gbar) = c$. Thus $\chi(G) + \chi(\gbar) = b+c = |V(G)|$,  contradicting the assumption that   $G$ is an NG-graph.
\end{proof}

We next present an algorithm for determining whether a graph is an NG-graph and in the affirmative case, computing its chromatic number.  The proof of correctness and an analysis of the complexity are given in Theorem~\ref{alg-correct}.

\noindent
{\bf Algorithm: NG}

\noindent
Input:  A graph $G$ with $n = |V(G)|$. 

\noindent
Output:  A determination of whether $G$ is an NG-graph and if so, its chromatic number.

\noindent
Initialize $k = 1$.

Loop: 
 Partition $V(G)$ according to vertex degrees as follows:

$A = \{v \in V(G): deg(v) = k-1\}$
 
$B = \{v \in V(G): deg(v) > k-1\}$
 
$C = \{v \in V(G): deg(v) < k -1\}$.  

Consider   questions  (i) -- (v). % until a `no' answer is received.
If the answer to any of the questions is `no', continue to step $(\ast)$.  Otherwise (if all answers are yes), go to step (vi).

\smallskip
\noindent
 (i) Is $G[A]$ a complete graph, an independent set or a 5-cycle?
 
 \smallskip
\noindent
 (ii) Is $G[B]$ complete?
 
 \smallskip
\noindent
 (iii) Is $G[C]$ an independent set?
 
 \smallskip
\noindent
 (iv) Is $ab \in E(G)$ for all $a\in A$ and all $b \in B$?
 
 \smallskip
\noindent
 (v) Is $ac  \not\in E(G)$ for all $a\in A$ and all $c \in C$?
 
 If the answers to (i) -- (v)  are `yes', then check if $\chi(G) = k$ as follows: 
  
\noindent (vi)   For $G[A]$   a complete graph, check that  $|A| + |B| =k$.
For $G[A]$   an independent set, check that $  |B| + 1=k$.
 For $G[A]$   a 5-cycle, check that $  |B| + 3=k$.
 
 If (vi) is affirmative, then  graph $G$ is an NG-graph and $\chi(G) = k$ and   the algorithm ends.   If not, continue to ($\ast$). 
 
\noindent ($\ast$)  If $k < n$, increment $k:= k+1$ and return to the beginning of the loop.
If $k=n+1$, then graph  $G$ is not an NG-graph.

\begin{Thm}
\label{alg-correct}
Algorithm NG determines whether graph $G$ is an NG-graph in polynomial time.
\end{Thm}

\begin{proof}
We first establish correctness.   We know $\chi(G)$ is between 1 and $n$, so start with $k = 1$, thinking of $k$ as a potential value of $\chi$.  
  
If the answers to (i) -- (vi) are all `yes', then Theorem~\ref{NG-equal-thm} ensures that  $G$ is a  an NG-graph, where (vi) verifies that $\chi(G) = k$.     
  If any of the answers to (i) -- (vi) are no, we try the next possible value of $k$.  If we reach $k=n+1$, then we have tried all possible values of $\chi(G) $, and $V(G)$  can not be partitioned so that $G$ has the necessary form.  Thus the algorithm correctly determines whether  $G$ is an NG-graph, and if so, computes the chromatic number.
  
  Each of the questions can be answered in time $O(n^2)$, and  we potentially have to increment $k$ from 1 to $n$ so the running time of the algorithm is $O(n^3)$.
\end{proof}

 \section{Characterizing those NG-graphs that are NGD-graphs}
 
In Section 3, we characterized NG-graphs and according to Theorem~\ref{NG-charac}, there are three possibilities for $A_G$.  We name them for convenience in the next definition.

\begin{Def} An NG-graph $G$ is of Type 1 if $G[A_G]$ is a clique, Type 2 if $G[A_G]$ is an independent set and Type 3 if $G[A_G]$ is a 5-cycle. 
\end{Def}

Note that an NG-graph with $|A_G|=1$ is both Type 1 and Type 2.
 
Analogously we would like to characterize NGD-graphs.  The set of NG-graphs and the set of NGD-graphs intersect, but neither is contained in the other, as demonstrated in Table \ref{small-table}.

\begin{table}[htbp] 

\begin{tabular}{|c|c|c|} \hline
$G$ & NG-graph & NGD-graph \\ \hline
$K_{3,1,1}$ & yes& yes\\ \hline 
$K_{3,2} $ & no & yes \\ \hline
$C_5$ & yes & no \\ \hline
$C_7$ & no & no \\\hline
\end{tabular} 
\bigskip

\caption{Examples of graphs that separate the classes of NG-graphs and NGD-graphs}
\label{small-table}
\end{table}
In this section we make progress toward this goal by characterizing the NG-graphs that are also NGD-graphs.  We show in Theorem~\ref{type3-NG}
that none of the Type 3 NG-graphs are NGD-graphs and in Theorem~\ref{type1} and Corollary \ref{type2}, we characterize those  Type 1 and Type 2 NG-graphs that are NGD-graphs.   The next result shows that the complement of a Type 1 NG-graph is a Type 2 NG-graph. 
 
\begin{Prop} \label{comps} If $G$ is an NG-graph, then $\gbar$ is an NG-graph.  Moreover, an NG-graph $G$ is of Type 1 iff \ $\gbar$ is of Type 2, and $G$ is of Type 3 iff \ $\gbar$ is of Type 3.  
\end{Prop}

\begin{proof}  The first sentence follows immediately from Definition \ref{NG-def}, so we focus on the statements in the second sentence. 
Since $G$ is an NG-graph,  $\chi(G)+\chi(\gbar)=n+1$ and $A_G$ is  the set of vertices whose degree is $\chi(G)-1$.  Therefore $A_G$ is the set of vertices whose degree in $\gbar$ is $(n-1)-(\chi(G)-1)= n-\chi(G) =\chi(\gbar)-1.$  Hence $A_G=A_{\gbar}$, and $G[A_G]$ is a clique if and only if $G[A_{\gbar}]$ is an independent set.  So $G$ is a Type 1 NG-graph if and only if $\gbar$ is a Type 2 NG-graph.  
The complement of a 5-cycle is a 5-cycle, so $G$ is a Type 3 NG-graph if and only if $\gbar$ is a Type 3 NG-graph.  
\end{proof}

In proving that a Type 3 NG-graph $G$ is not an NGD-graph, it will be helpful to have an optimal coloring of $G$ in which one color appears on only one vertex.  This is possible by our next lemma.  

 \begin{lemma} 
 \label{AB-lem} (i) If $G$ is a Type 1 NG-graph and $x \in A_G \cup B_G$ then there exists a proper coloring of $G$ using $\chi(G)$ colors in which   vertex $x$ is uniquely colored. (ii) If $G$ is a Type 2 NG-graph and $x \in B_G$ then  there exists a proper coloring of $G$ using $\chi(G)$ colors in which   vertex $x$ is uniquely colored.
 
 \end{lemma}
 \begin{proof}  We start by proving (i).  Let $G$ be  a Type 1 NG-graph and fix a proper coloring of $G$ using $\chi(G)$ colors.
 Since the vertices in  $A_G \cup B_G$ induce a clique in $G$, they    are all colored distinctly.  If there are  other vertices with $x$'s color, they must be in $C_G$.    First consider the case in which $x \in B_G$.  Since $A_G \neq \emptyset$, there exists $y \in A_G$, and all vertices of $C_G$  may be recolored to have $y$'s color.  This leaves $x$ as the only vertex in its color class.
 
 Next consider the case in which $x \in A_G$ and  $|A_G| = 1$.   By the definition of Type 1 NG-graphs, we know $N(x) = B_G$ and for each $c\in C_G$ we know $deg(c) \neq deg(x) = |B_G|$.  Thus each $c \in C_G$ has a non-neighbor in $B_G$.  We can recolor each $c \in C_G$ to be the color of any of its non-neighbors in $B_G$.  This leaves $x$ as the only vertex in its color class.
 
 Finally,  consider the case in which $x \in A_G$ and  $|A_G| > 1$.  Let $a \in A_G$ where $a \neq x$.  Since $ax \in E(G)$ we know $a$'s color is different from $x$'s color.  Then each vertex in $C_G$ can be colored with $a$'s color and this leaves $x$ as the only vertex in its color class.
 
 The proof of (ii) is similar to the first paragraph of the proof of (i).

 \end{proof}
 
 \begin{Thm}
 \label{type3-NG}
Type 3 NG-graphs are not NGD-graphs.
 \end{Thm}
 \begin{proof}
 Let $G$ be a Type 3 NG-graph, and for a contradiction, assume $G$ is also  an NGD-graph.  
 Fix a distinguishing coloring of $G$ using colors $1,2,3, \ldots ,D(G)$. By definition of a Type 3 NG-graph, the vertices in $A_G$ induce a 5-cycle in $G$, which we represent by $A_G = \{v_1,v_2,v_3,v_4,v_5\}$ with adjacencies $v_1v_2, v_2v_3, v_3v_4, v_4v_5, v_5v_1$. 
 
  By the structure of Type 3 NGD-graphs, we know that any particular 
$v \in B_G \cup C_G$  has the same relationship to 
each vertex in $A_G$.  Furthermore, any automorphism of $G$ preserves the sets $A_G, B_G, C_G$ because these sets are defined in terms of vertex degrees.  Thus a     coloring of $G$ is distinguishing if and only if it is distinguishing on both  $G[A_G]$ and $G[B_G \cup C_G]$.  Since $D(C_5) = 3$, we may recolor the vertices in $A_G$ so that they use three colors and without loss of generality, vertices $v_1,v_3$ get color 1, vertices $v_2,v_4$ get color 2, and vertex $v_5$ gets color 3.  Let    $V_i$ be the vertices of color $i$ and let $G_i = G[V_i]$ be the graph induced by the vertices of color $i$. 

For simplicity of notation in what follows, let $H = G_1$ and $F = G_2$.  
By Corollary~\ref{H-F-cor}, graphs $H$ and $F$ are NG-graphs.  Furthermore, $H$ (and likewise $F$) is not of Type 3, since there is no $C_5$ induced in $H$ (or in $F$).    Thus $H$ and $F$ are Type 1 or Type 2 NG-graphs and by Proposition \ref{comps}, so are their complements $\Hbar$ and $\Fbar$.     Since $v_1$ and $v_3$ have the same degree in $\Hbar$, they are in the same part of the $ A_{\Hbar}   \cup  B_{\Hbar} \cup  C_{\Hbar}$ partition  of $V(\Hbar)$.  
 If $\Hbar$ is Type 1,  since  $v_1$ is adjacent to $v_3$ in $\Hbar$, we know $v_1, v_3  \in  A_{\Hbar}   \cup  B_{\Hbar}$.  If $\Hbar$ is Type 2,  since  $v_1$ is adjacent to $v_3$ in $\Hbar$, we know $v_1, v_3  \in  B_{\Hbar}$.    Now applying Lemma~\ref{AB-lem} to $\overline{H}$ (and $\overline{F}$), we conclude that  there exists a proper coloring of $\Hbar$  using $\chi(\Hbar)$ colors in which $v_1$ is uniquely colored.    Similarly, there exists a proper coloring of $\Fbar$  using $\chi(\Fbar)$ colors in which $v_2$ is uniquely colored.
 
Following the proof of Theorem \ref{NG-new},  for each $i = 1,2,3,  \ldots,  D(G)$, 
 we create a  new coloring of   $G_i$  using $\chi(G_i)$ colors and   of $\overline{G_i}$  using $\chi(\gbar_i)$ colors,   so that   $\chi(G_i) + \chi(\overline{G_i}) = |V_i| + 1$.  Note that in this coloring we use a different palette of colors for each $G_i$ and for each $\overline{G_i}$.    Furthermore, we choose colorings of $\Hbar = \overline{G_1}$ and $\Fbar = \overline{G_2}$ so that $v_1$ is uniquely colored (yellow) in $\Hbar$ and $v_2$  is  uniquely colored (purple)  in $\Fbar$.    Finally, we switch $v_2$'s color to yellow.  Since $v_1$ and $v_2$ are not adjacent in $\gbar$,   the new coloring is proper.  It is also distinguishing as follows:
 
Recall that in $G $ or $\gbar$, any automorphism preserves the set of vertices $\{v_1, v_2, v_3, v_4,$ $ v_5\}$.  Note that $v_1,v_2,v_3,v_4,v_5$ were all given different colors in $\gbar$ before this final switch of $v_2$'s color, since they come from three different palettes, and $v_1, v_2$ were uniquely colored.  After the switching $v_2$'s color, each of $v_3, v_4, v_5$ is fixed by every automorphism that preserves the colors.  There is no automorphism that switches $v_1$ and $v_2$ and preserves the colors after the final switch since $v_1$ is adjacent to $v_3$ and $v_2$ is not adjacent to $v_3$ in $\gbar$.    There are no vertices outside of $A_{\gbar}$ that needed $v_1$ and $v_2$ to distinguish them, since all vertices outside of $A$ have the same relationship to each vertex inside of $A$.  
 
Now we have given   colorings of $G$ and $\gbar$  that are distinguishing and proper using a total of $n + D(G) - 1$ colors, contradicting $G$ being an NG-graph.
 \end{proof}

We will need a refinement on the vertex partition of a Type 1 NG-graph, where we further partition the set $C_G$ as $L_{G}\cup M_{G}$.  
\begin{Def} \rm \label{ABLM}  Let $G$ be a Type 1 NG-graph with $ABC$-partition $A_G\cup B_G\cup C_G$.  We define the $ABLM$-partition of $G$ to be $V(G)=A_G\cup B_G\cup L_{G}\cup M_{G}$ where $L_{G}=\{v\in C_G:\mbox{$v$ is adjacent to every vertex in $B_G$}\}$. When it is unambiguous we write $A=A_G, B=B_G, L=L_{G}$ and $M=M_{G}$. 
%Let $a=|A_G|$,  $b=|B_G|  $, $e=|C_{1,G}|$, $c_2=|C_{2,G}|$.  
\end{Def}
From  Definition \ref{ABLM}, we know $deg(v)=|B_G|$ for each $v\in L_G$, and $deg(v)\leq |B_G|-1$ for each $v\in M_G$.  If $|A_G|=1$, then $L_G=\emptyset$ because any $v\in L_G$ would have the same degree as the vertex in $A_G$, contradicting Definition \ref{ABC-def}. 
%Hence, if $|A_G|=1$, then $e=0$.  

Let $\oplus$ be the external direct product of groups, and recall that $S_n$ is the group of permutations of the set $\{1,2,3, \ldots, n\}$. 

 \begin{Prop} \label{type1-gamma} Let $G$ be a Type 1 NG-graph and 
 define $A, B, L, M$ as in Definition \ref{ABLM}.  Then 
 \[Aut(G) \cong S_{|A|} \oplus S_{|L|}\oplus \Gamma\] 
 where $\Gamma$ is the subgroup of automorphisms that act on $G[B\cup M]$ and fix $A$ and $L$.
\end{Prop}

\begin{proof} 
%Let $A=A_G, B=B_G, L=L_G$ and $M=M_G$. 
We note that any automorphism of $G$ preserves the sets $A,B,L,M$, because of the different vertex degrees in each set.   Thus, in a distinguishing coloring, we may use the same set of colors for each set of vertices.  Further, since each vertex in $A$ has the same set of neighbors outside of $A$, 
and each vertex in $L$ has the same neighborhood, then the action of $Aut(G)$ is independent on the three subgraphs, $G[A], G[L], G[B\cup M]$.  Then 
the automorphism group of $G$ is isomorphic to $Aut(G[A])\oplus Aut(G[L])\oplus \Gamma$ where $\Gamma\subseteq Aut(G)$ is the subgroup of automorphisms that act on $G[B\cup M]$ and fix $A$ and $L$.  Since $A$ is complete and $L$ is an independent set, $Aut(A)=S_{|A|}$ and $Aut(L)=S_{|L|}$.  Hence 
$ Aut(G)\cong S_{|A|}\oplus S_{|L|}\oplus \Gamma.$
\end{proof}

\begin{Cor} \label{group} Let $G$ be a Type 1 NG-graph and define $A, B, L,M $ as in Definition \ref{ABLM}.   Let $a=|A|$ and $\ell=|L|$. Then 
\[D(G)=max\{a,\ell, D^{\Gamma}(G[B\cup M])\}\] where $\Gamma$ is the subgroup of automorphisms that act on $G[B\cup M]$ and fix $A$ and $L$.
\end{Cor}
 
\begin{proof} Recall the definition of the distinguishing number with respect to $\Gamma$ from Definition \ref{DG}.  In any distinguishing coloring of $G$, the colors can be reused for each set in the $ABLM$-partition of $G$, and the number of colors needed for $A$ is $a$ and for $L$ is $\ell$, thus
$D(G)=max\{a,\ell, D^{\Gamma}(G[B\cup M])\}$.
\end{proof}

We now define some necessary parameters, $x_G$ and $y_G$. 
\begin{Def} \rm \label{xandy} Let $G$ be a Type 1 NG-graph and define $A, B, L, M$ as in Definition \ref{ABLM}. Let $b=|B|$,  $m=|M|$ and let $\Gamma$ be defined as in Proposition \ref{type1-gamma}.  We define $x=x_G$ to be the minimum number of colors, above the $b$ colors used on the vertices in $B$, needed to color the vertices in $M$, to get a distinguishing and proper coloring of $G[B\cup M]$ under the action of $\Gamma$.  
Similarly, we define  $y=y_G$ to be the minimum number of colors, above the $m$ colors used on the vertices in $M$, needed to color the vertices in $B$ to get a distinguishing and proper coloring of $\gbar[B\cup M]$ under the action of $\Gamma$. 
\end{Def}  The next lemma gives a bound on $x+y$, and following it are more two technical lemmas. 

\begin{lemma}  \label{xy} Let $G$ be a Type 1 NG-graph and define $A, B, L, M$ as in Definition \ref{ABLM}. Then \[x, y\leq x+y\leq D(G).\]
\end{lemma}

\begin{proof} The first inequality follows immediately because $x, y$ are both nonnegative.  For the second inequality, let $b=|B|$ and $m=|M|$.  By the definition of $x$,  $\chi_D^{\Gamma}(G[B\cup M])=b+x$, and by the definition of $y$,  $\chi_D^{\Gamma}(\gbar[B\cup M])=m+y$.  Then, applying 
Proposition \ref{NG-new-r} to the graph $G[B\cup M]$, we get  \[(b+x)+(m+y)=\chi_D^{\Gamma}(G[B\cup M]) + \chi_D^{\Gamma}(\gbar[B\cup M]) \leq b+m + D^{\Gamma}(G[B\cup M]).\]
Furthermore, $D(G)\geq  D^{\Gamma}(G[B\cup M]) $ from Corollary \ref{group}.  Thus, $x+y\leq D(G)$ as desired.
%  The second inequality follows from Corollary \ref{group}.
\end{proof}

\begin{lemma} \label{x} Let $G$ be a Type 1 NG-graph and define $A, B, L, M$ as in Definition \ref{ABLM}, and $x$ as in Definition \ref{xandy}.  Then  $x< D(G)$.
\end{lemma}

\begin{proof} Let $b=|B|$.  
We define a distinguishing and proper coloring  of $G[B\cup M]$. First, color the vertices in $B$ with the $b$ colors $\{1,2,3,\ldots, b\}$.  We know $b$ colors are needed, since $G[B]$ is complete.  Each vertex of $B$ is then fixed by its unique label.  The vertices in $M$ form an independent set, so would need to be different colors only in order to be distinguished.  If $u,v\in M$ have different neighborhoods in $B$, then the colors of the vertices in those neighborhoods are different sets, so an automorphism taking $u$ to $v$ will not preserve colors.  

For each vertex $u\in M$, let $S_u\subseteq \{1,2,3,\ldots, b\}$ be the set of colors of the vertices in $N(u)$, and 
define $T_u=\{ v\in M: N(v)=N(u)\}$. Note that for any two vertices in $T_u$, there  is an automorphism of $G$ that interchanges them and fixes the rest of $G$.  Thus, $D(G)\geq |T_u|$ for every $u\in M$.  
In order to achieve a proper and distinguishing coloring, for each $u\in M$, each set $T_u$ must be colored distinctly and the colors used on the vertices in $T_u$ must be disjoint from $S_u$.  Conversely, if this is achieved, we have a proper and distinguishing coloring of $G[B\cup M]$.  Let $u_1, u_2, \ldots, u_k$ be chosen so that $T_{u_1}, T_{u_2}, \ldots , T_{u_k}$ is a partition of $M$.  

For $1\leq i\leq k$, we color the vertices in $T_{u_i}$ distinctly, using as many colors in $\{1,2,3,\ldots,b\} - S_u$ as possible.  The smallest number of colors, in addition to our original $b$ colors, that we need is $max_{1\leq i\leq k} \{|T_{u_i}| - (b-|S_u|)\}$.  Since $x$ 
is the minimum number of colors, above the $b$ colors used on the vertices in $B$, needed to color the vertices in $M$, to get a distinguishing and proper coloring of $G[B\cup M]$ under the action of $\Gamma$, $x=max_{1\leq i\leq k} \{|T_{u_i}| - (b-|S_u|)\}$. 
By the definition of $M$, each vertex in $M$ is missing at least one edge to $B$, so  for each $i$, we can use at least one color from $\{1,2,3\ldots,b\}$, and thus $x<max_{1\leq i\leq k} \{|T_i|\}\leq D(G)$.  

\end{proof}

\begin{lemma} \label{y} Let $G$ be a Type 1  NG-graph $G$, with an ABLM partition.   If $\ell=0$, then $y< D(G)$.
\end{lemma}

\begin{proof} Given $\ell=0$, we know $C=M$.  
Since each vertex in $B$ has degree greater than each vertex in $A$ in $G$, each vertex in $B$ has an edge to some vertex in $M$ in $G$.  That means that in $\gbar$, every vertex in $B$ is missing an edge to some vertex in $M$.  For each $v\in B$, let $W_v=\{w\in B:N_{\gbar}(w)=N_{\gbar}(v) \}$.  
Following the argument of Lemma \ref{x}, $y< max_{v\in B}\{|W_v|\}\leq  D(\gbar)=D(G)$.
\end{proof}

We are now ready to characterize those Type 1 NG-graphs that are also NGD-graphs.  
Note that the vertices in $A_G$ in a Type 1 NG-graph $G$ form a complete subgraph, and all have the same neighborhood in the rest of the graph.  So $D(G)\geq |A_G|$.  Similarly, an independent set $L$ of vertices in $C_G$ each of which is adjacent to every vertex in $B_G$ would all need to be distinctly colored in any distinguishing coloring of $G$, so $D(G)\geq |L|$.  In the theorem below, we show that  for any Type 1 NG-graph $G$, $D(G)$ is the maximum of these quantities if and only if $G$ is NGD-graph.

 \begin{Thm} \label{type1}  Let $G$ be a Type 1  NG-graph $G$ and $a=|A_G|$ and $\ell$ is the number of vertices in $C_G$ that are adjacent to every vertex in $B_G$.
Then $G$ is an NGD-graph iff $D(G)=max\{a,\ell\}$.   \end{Thm}

\begin{proof}
In any distinguishing and proper coloring  of $G$, all vertices in $B$ are colored distinctly, since $B$ is complete, and there are no edges between the set $A$ and the set $L\cup M$, 
so colors used on vertices in $A$ can be reused in $L\cup M$.  There are  no edges between $L$ and $M$, so colors used on $L$ can be reused on $M$.  Thus, $\chi_D(G)=b+max\{a, \ell, x\}$.  In a distinguishing and 
proper coloring  of $\gbar$, every vertex in $A\cup L\cup M$ must be colored 
distinctly since $L\cup M$ is a complete graph, and all vertices in $A$ are adjacent to 
all vertices in $L\cup M$, and all vertices in $A$ must be colored distinctly to eliminate 
the symmetries.  Colors used on $A$ and $L$ can be re-used on $B$, since there are no 
edges in $\gbar$ between $B$ and $A\cup L$.  Thus, $\chi_D(\gbar) = m+max\{a+\ell, y\}$. 
Hence 

\begin{equation}\label{eqn4} \chi_D(G)+\chi_D(\gbar)=b+\mbox{max}\{a,\ell,x\}+m+\mbox{max}\{a+\ell, y\}.\end{equation}

\medskip

We analyze the cases, depending on $max\{a,\ell,x\}$ and $max\{a+\ell, y\}$. Recall from Definition \ref{NG-def} that $G$ is an NGD-graph iff $\chi_D(G)+\chi_D(\gbar)=a+b+\ell+m+D(G).$  Our proof will also show that the graphs in Cases (2) - (5) are not NGD-graphs.

\medskip 

\noindent {\bf Case (1)} $max\{a,\ell,x\}=max\{a,\ell\}$ and $max\{a+\ell,y\}=a+\ell$. 

Using Equation~(\ref{eqn4}), we have 
\[\chi_D(G)+\chi_D(\gbar)=b+max\{a,\ell\}+m+a+\ell=(a+b+\ell+m)+max\{a,\ell\}.\]    
In this case $G$ is an NGD-graph if and only if $D(G)=max\{a,\ell\}$.  By Corollary~\ref{group}, $D(G)\geq max\{a,\ell\}$, so if $D(G)=max\{a,\ell\}$,  then $G$ is an NGD-graph, and if $D(G)>max\{a,\ell\}$, then $G$ is not an NGD-graph.  

\medskip
In Cases (2) - (5), we show that $G$ is not an NGD-graph and that $D(G)>max\{a,\ell\}$. 
\medskip

\noindent {\bf Case (2)} $max\{a,\ell,x\}=x$ and $max\{a+\ell,y\}=a+\ell$. 

Using Equation~(\ref{eqn4}), we have 
\[\chi_D(G)+\chi_D(\gbar)=b+x+m+a+\ell=(a+b+\ell+m)+x.\] 
Then $G$ is an NGD-graph iff $D(G)=x$.  By Lemma \ref{x}, $x<D(G)$, so $G$ is not  an NGD-graph, and indeed $D(G)>x\geq max\{a, \ell\}$ as desired.

\medskip

\noindent {\bf Case (3)} $max\{a,\ell,x\}=a$ and $max\{a+\ell,y\}=y$.  

Using Equation~\ref{eqn4}, we have
\[\chi_D(G)+\chi_D(\gbar)=b+a+m+y=(a+b+m)+y.\]  
  Suppose that $G$ were an NGD-graph.  Then $y=\ell+D(G)$, by Lemma \ref{xy}, $\ell + D(G)=y\leq x+y\leq D(G).$ So $\ell=0$ and $y=D(G)$.  By Lemma \ref{y}, when $\ell=0$, we have $y<D(G)$, a contradiction, so $G$ is not an NGD-graph.  

Using Lemma \ref{xy} and the assumptions of this case, we have $D(G)\geq  x+y\geq y \geq a+\ell$. We know $a>0$ by Theorem \ref{NG-charac}. If $\ell>0$, then $D(G)\geq a+\ell>max\{a,\ell\}$.  If $\ell=0$, then by Lemma \ref{y}, $D(G)>y\geq a+\ell=a\geq max\{a,\ell\}$.  So we have shown $D(G)>max\{a,\ell\}$ as desired.  

 \medskip

\noindent {\bf Case (4)} $max\{a,\ell,x\}=\ell$ and $max\{a+\ell,y\}=y$. 

Using Equation~\ref{eqn4}, we have
\[\chi_D(G)+\chi_D(\gbar)=b+\ell+m+y.\]  
Using $a>0$ from Theorem \ref{NG-charac} and $y\leq D(G)$ from Lemma \ref{xy}, we have
$\chi_D(G)+\chi_D(\gbar)<(a+b+\ell+m)+y\leq (a+b+\ell+m)+D(G)$.  Thus, $G$ is not an NGD-graph.  

Using the assumptions of this case, and Theorem \ref{NG-charac}, we have $\ell\geq a>0$, thus $a+\ell>max\{a,\ell\}$.  Now using Lemma \ref{xy}, and the assumptions of this case, we have $D(G)\geq  y\geq a+\ell >max\{a,\ell\}$ as desired.   
  \medskip

\noindent {\bf Case (5)} $max\{a,\ell,x\}=x$ and $max\{a+\ell,y\}=y$.  

Using Equation~\ref{eqn4}, we have 
\[\chi_D(G)+\chi_D(\gbar)=b+x+m+y.\]
 Using $a>0$ from  Theorem \ref{NG-charac}  and $x+y\leq D(G)$ from  Lemma \ref{xy}, 
we have $\chi_D(G)+\chi_D(\gbar)< (a+b+\ell+m)+D(G).$  Hence $G$ is not an NGD-graph.  

Using $a>0$ from Theorem \ref{NG-charac} and the assumptions of this case, we have 
$x+y \geq a+a+\ell >max\{a,\ell\}$.  However, $D(G)\geq x+y$ by Lemma \ref{xy}, so we get $D(G)>max\{a,\ell\}$ as desired.  
\end{proof}

If $G$ is a Type 2 NG-graph, then its complement, $\gbar$, is a Type 1 NG-graph by Proposition \ref{comps}, and applying Theorem \ref{type1} to $\gbar$ yields the following. 

\begin{Cor} \label{type2}
A Type 2 NG-graph $G$ is an NGD-graph iff   $D(G) =|A_G|$ or $D(G)$ equals the number of vertices in $B_G$ that have no adjacencies to vertices in $C_G$.
\end{Cor}

In our next example, we describe a Type 1 NG-graph which falls into Case 1 of the proof of Theorem \ref{type1}, but is not an NGD-graph. 
\begin{example} \rm 
Let $G$ be a Type 1 NG-graph with $a=1$, $b=5$, $\ell=0$, $m=5$, and then define the edges between $B_G$ and $M_G$ so that each vertex in $M_G$ has degree 1 and  is adjacent to a different vertex in $B_G$.  Then $D(G)=3$, $x=0$, $y=0$, which fits in Case 1, except that $G$ is not an NGD-graph, because $D(G)>max\{a,\ell\}$.  
\end{example}

We conclude with two questions and an acknowledgement.  
\begin{Ques} \rm Theorem \ref{type1} characterizes Type 1 NG-graphs which are NGD-graphs by their distinguishing number.  Can the distinguishing number of the class of Type 1 NG-graphs be determined in polynomial time?
\end{Ques}
\begin{Ques} \rm 
In Theorem \ref{type1} and Corollary \ref{type2}, we have characterized those NG-graphs which are NGD-graphs.  Can this be extended to a characterization of the class of NGD-graphs? 
\end{Ques}

\subsection*{Acknowledgements} The authors would like to thank Galen Turner, who suggested we consider distinguishing chromatic number analogues of the Nordhaus and Gaddum inequalities.

\end{document}